\documentclass[runningheads, orivec, envcountsame]{llncs}
\usepackage{amsmath,amssymb}
\usepackage{array}
\usepackage{mathrsfs}
\usepackage{stmaryrd}
\usepackage{bm}
\spnewtheorem*{theorem*}{Theorem}{\bf}{\itshape}
\spnewtheorem*{proof*}{}{\itshape}{\upshape}
\newcolumntype{C}[1]{>{\hfil}m{#1}<{\hfil}}
\DeclareMathOperator{\FV}{FV}
\DeclareMathOperator{\dom}{dom}
\newcommand{\To}{\Rightarrow}
\newcommand{\bzero}{\mathbf{0}}
\newcommand{\bone}{\mathbf{1}}
\newcommand{\Fml}{\mathop{\mathrm{FOFml}}}
\newcommand{\ba}{\mathbf{a}}
\newcommand{\bb}{\mathbf{b}}
\newcommand{\bc}{\mathbf{c}}

\newcommand{\ttfunc}[1]{\mathop{\mathrm{t}_{#1}}}
\newcommand{\arity}{\mathop{\mathrm{ar}}}
\newcommand{\ILS}{\mathop{\mathrm{ILS}}}

\newcommand{\CLS}{\mathop{\mathrm{CLS}}}
\newcommand{\sK}{\mathscr{K}}
\newcommand{\sC}{\mathscr{C}}
\newcommand{\sM}{\mathscr{M}}
\newcommand{\Sqt}{\mathop{\mathrm{Sqt}}}
\newcommand{\PropFml}{\mathop{\mathrm{Fml}}}
\newcommand{\FOCLS}{\mathop{\mathrm{FOCLS}}}
\newcommand{\FOCDS}{\mathop{\mathrm{FOCDS}}}
\newcommand{\FOILS}{\mathop{\mathrm{FOILS}}}
\newcommand{\FOCL}{\mathop{\mathrm{FOCL}}}
\newcommand{\FOCD}{\mathop{\mathrm{FOCD}}}
\newcommand{\FOIL}{\mathop{\mathrm{FOIL}}}

\newcommand{\FOSqt}{\mathop{\mathrm{FOSqt}}}
\newcommand{\emptyfunc}{\varnothing}
\newcommand{\verticalspace}{\vspace{0.1cm}}
\begin{document}
\title{Effect of the Choice of Connectives \\ on the Relation  between \\ the Logic of Constant Domains \\ and Classical Predicate Logic}
\titlerunning{Effect of the Choice of Connectives}
\author{Naosuke Matsuda\inst{1}
\and
Kento Takagi\inst{2}\orcidID{0000-0003-3810-9610}}
\authorrunning{N. Matsuda and K. Takagi}
\institute{Department of Engineering, Niigata Institute of Technology, \\
Fujihashi, Kashiwazaki City, Niigata 945-1195, Japan \\
\email{matsuda.naosuke@gmail.com}\\
\and
Department of Computer Science, Tokyo Institute of Technology, \\ Ookayama, Meguro-ku, Tokyo 152-8522, Japan \\
\email{kento.takagi.aa@gmail.com}}
\maketitle         
\begin{abstract}
It is known that not only classical semantics but also intuitionistic Kripke semantics can be generalized so that it can treat arbitrary propositional connectives characterized by truth tables, or truth functions. In our previous work, it has been shown that the set of Kripke-valid propositional sequents and that of classically valid propositional sequents coincide if and only if all available propositional connectives are monotonic. The present paper extend this result to first-order logic showing that, in the case of predicate logic, the condition that all available propositional connectives are monotonic is a necessary and sufficient condition for the set of sequents valid in all constant domain Kripke models, not the set of Kripke-valid sequents, and the set of classically valid sequents to coincide. 

\end{abstract}
\section{Introduction}\label{Section 1}
\subsection{Generalized propositional logic}\label{Subsection 1.1}
In \cite{kripke1965semantical}, Kripke provided the intuitionistic interpretation for formulas built out of the usual propositional connectives $\lnot$, $\to$, $\land$ and $\lor$. The notion of validity in intuitionistic logic can be defined with this interpretation. Rousseau~\cite{rousseau1970sequents} and Geuvers and Hurkens~\cite{geuvers2017deriving} extended the intuitionistic interpretation so that it can treat arbitrary propositional connectives characterized by truth tables, or truth functions. Their idea is very simple: when $c$ is a propositional connective and $\ttfunc{c}$ is the truth function associated with $c$, then the interpretation $\| c(\alpha_1, \ldots, \alpha_n) \|_w$ of formula $c(\alpha_1, \ldots, \alpha_n)$ at world $w$ is defined as follows:
  \[
  \| c(\alpha_1, \ldots, \alpha_n )\|_w = 1 \ \text{ if and only if } \ \text{$\ttfunc{c}(\| \alpha_1 \|_v, \ldots, \| \alpha_n \|_v) = 1$ for all $v \succeq w$}.
  \]   
It is well-known that the relation between intuitionistic logic and  classical logic changes by the choice of propositional connectives. In particular, the relation between the sets of valid sequents changes. 
For example, $\ILS(\{\lnot\}) \subsetneq \CLS(\{ \lnot \})$ and $\ILS(\{\land, \lor \}) = \CLS(\{ \land, \lor \})$, where, for a set of propositional connectives $\sC$, $\ILS(\sC)$ denotes the set of Kripke-valid propositional sequents built out of the connectives in $\sC$ and $\CLS(\sC)$ denotes the set of classically valid propositional sequents built out of the connectives in $\sC$.
Then, there arises a natural question: for what $\sC$, does $\ILS(\sC) = \CLS(\sC)$ hold? We answered this question in \cite{kawano2021effect}. 
But, before describing the answer, we briefly review some necessary notions.

For each connective $c$, $\arity(c)$ denotes the arity of $c$.
Let $\sC$ be a set of propositional connectives. We denote by $\ILS(\sC)$ the set of Kripke-valid sequents built out of the propositional connectives in $\sC$ and by $\CLS(\sC)$ the sets of classically-valid sequents built out of the propositional connectives in $\sC$. For a sequence of truth values $\ba \in \{ 0, 1 \}^n$, $\overline{\ba} \in \{0,1\}^n$ denotes the sequence of truth values obtained from $\ba$ by inverting $0$ and $1$. $\sqsubseteq_n$ is the natural order on $\{ 0, 1 \}^n$, that is, for $\ba = \langle a_1, \ldots, a_n \rangle \in \{0,1\}^n$ and $\bb = \langle b_1, \ldots, b_n \rangle \in \{0, 1 \}^{n}$, $\ba \sqsubseteq_n \bb$ if and only if $a_i \leq b_i$ for all $i = 1, \ldots, n$. For $\ba \in \{0,1\}^n$ and $\bb \in \{0, 1\}^n$, $\ba \sqcap \bb$ denotes the infimum of the set $\{ \ba, \bb \}$ with respect to $\sqsubseteq_n$. $\langle 1, \ldots, 1 \rangle \in \{ 0, 1 \}^n$ and $\langle 0, \ldots, 0 \rangle \in \{ 0, 1 \}^n$ are denoted by $\bone_n$ and $\bzero_n$, respectively. We shall omit the subscript $n$ of $\sqsubseteq_n$, $\bone_n$ and $\bzero_n$ if it is clear from the context. For details, see \S \ref{Section 2}.

Then, the necessary and sufficient condition for $\ILS(\sC)$ and $\CLS(\sC)$ to coincide is described as follows:  
  \begin{theorem*}[\cite{kawano2021effect}]
  $\ILS(\sC) = \CLS(\sC)$ if and only if all connectives in $\sC$ are monotonic, that is, all $c \in \sC$ satisfy the following condition: for any $\ba, \bb \in \{ 0, 1\}^{\arity(c)}$, if $\ba \sqsubseteq \bb$ then $\ttfunc{c}(\ba) \leq \ttfunc{c}(\bb)$.
  \end{theorem*}
\subsection{Results}\label{Subsection 1.2}
The present paper extends the preceding theorem to first-order logic. Generalized Kripke semantics can be extended to first-order logic by adding $\forall$ and $\exists$ with the usual interpretations. 
Let $\FOILS(\sC)$ denote the set of Kripke-valid sequents built out of the quantifiers $\forall$ and $\exists$ and the propositional connectives in $\sC$ and let $\FOCLS(\sC)$ denote the set of classically valid sequents built out of the quantifiers $\forall$ and $\exists$ and the propositional connectives in $\sC$. 
Then, the following claim might seem a straightforward extension of the preceding theorem to first-order logic: $\FOILS(\sC) = \FOCLS(\sC)$ if and only if all connectives in $\sC$ are monotonic. However, this claim fails. Instead, if we extend the proof of the preceding theorem, we obtain a necessary and sufficient condition for the set of sequents that are valid with respect to \emph{constant domain} Kripke semantics and that of classically valid sequents to coincide:
  \begin{theorem*}
  Let $\FOCDS(\sC)$ denote the set of sequents built out of the quantifiers $\forall$ and $\exists$ and the propositional connectives in $\sC$ which are valid in all constant domain Kripke models. Then,
  $\FOCDS(\sC) = \FOCLS(\sC)$ if and only if all connectives in $\sC$ are monotonic.
  \end{theorem*}
We give a proof of this main theorem extending the proof of the theorem that gives the necessary and sufficient condition for $\ILS(\sC)$ and $\CLS(\sC)$ to coincide.
\subsection{Overview}\label{Subsection 1.3}
In \S \ref{Section 2}, we introduce basic concepts and extend the general propositional logic to first-order logic. In \S \ref{Section 3}, we show the main theorem.

\section{Preliminaries}
\label{Section 2}
\subsection{Connectives and truth functions}
\label{Subsection 2.1}
The elements of a set $\{ 0, 1\}$ are called the \emph{truth values}. $\{ 0, 1 \}^n$ denotes the set of sequences of truth values of length $n$. We shall use letters $\ba$, $\bb$ and $\bc$ to denote arbitrary finite sequences of truth values. We denote by $\bzero_n$ and $\bone_n$ the sequence $\langle 0, \ldots, 0 \rangle \in \{0, 1\}^n$ and $\langle 1, \ldots, 1 \rangle \in \{0, 1\}^n$, respectively. 
For $\ba \in \{0,1\}^n$, we denote by $\ba[i]$ the $i$-th value of $\ba$. For example, $\langle 0, 1, 0 \rangle [1] = \langle 0, 1, 0 \rangle [3] = 0$ and $\langle 0, 1, 0 \rangle [2] = 1$.
For $\ba \in \{ 0, 1\}^n$, $\overline{\ba}$ denotes the sequence obtained from $\ba$ by inverting $0$ and $1$. For example, $\overline{\langle 0, 1, 0 \rangle} = \langle 1, 0, 1 \rangle$.
An $n$-ary \emph{truth function} is a function from $\{ 0, 1\}^n$  to $\{ 0, 1\}$. 

The natural order $\sqsubseteq_n$ on $\{ 0, 1\}^n$ is defined as follows: for $\ba \in \{ 0, 1 \}^n$ and $\bb \in \{ 0, 1 \}^n$, $\ba \sqsubseteq_n \bb$ if and only if $\ba[i] \leq \bb[i]$ for all $i = 1, \ldots, n$. Here, $\leq$ denotes the usual order on $\{0,1\}$ defined by $0 \leq 0$, $1 \leq 1$, $0 \leq 1$ and $1 \not \leq 0$. In what follows, we shall omit the subscript $n$ of $\bzero_n$, $\bone_n$ and $\sqsubseteq_n$, since it is clear from the context. For $\ba, \bb \in \{0,1\}^n$, $\ba \sqcap \bb$ denotes the infimum of $\{ \ba, \bb \}$. It is obvious that $\ba \sqcap \bb$ can be calculated as follows:
  \[
  (\ba \sqcap \bb) [i] = 
    \begin{cases}
    1 & \text{if $\ba[i] = 1$ and $\bb[i] = 1$} \\
    0 & \text{if $\ba[i] = 0$ or $\bb[i] = 0$}.
    \end{cases}
  \]
An $n$-ary truth function $f$ is said to be \emph{monotonic} if for all $\ba, \bb \in \{ 0, 1 \}^n$, $\ba \sqsubseteq \bb$ implies $f(\ba) \leq f(\bb)$. 

\subsection{Propositional connectives and formulas}
\label{Subsection 2.2}
A \emph{propositional connective} is a symbol with a truth function. For a propositional connective $c$, we denote by $\ttfunc{c}$ the truth function associated with $c$ and by $\arity(c)$ the arity of $\ttfunc{c}$. We shall use letters $c$ and $d$ as metavariables for propositional connectives.
 
Assume a set $\sC$ of propositional connectives is given. We define the first-order language with propositional connectives in $\sC$. It consists of the following symbols: countably infinitely many individual variables; countably infinitely many \footnote{As we can see from the proofs in this paper, only a small number of supplies of predicate symbols suffice actually.} $n$-ary predicate symbols for each $n \in \mathbb{N}$; propositional connectives in $\sC$; quantifiers $\forall$ and $\exists$. $0$-ary predicate symbols are also called  \emph{propositional symbols}.  
Although all arguments in this paper work with trivial modifications if the language has function symbols and constant symbols, we assume the language has no function symbols and no constant symbols for simplicity.
We shall use $x$, $y$ and $z$ as metavariables for individual variables; $p$, $q$, $r$ and $s$ for predicate symbols; $c$ and $d$ for propositional connectives. An \emph{atomic formula} is an expression of the form $p(x_1, \ldots, x_n)$, where $p$ is an $n$-ary predicate symbol. The set $\Fml(\sC)$ of \emph{(first-order) formulas} is defined inductively as follows:
  \begin{itemize}
  \item if $\alpha$ is an atomic formula, then $\alpha \in \Fml(\sC)$;
  \item if $c \in \sC$ and $\alpha_1, \ldots, \alpha_{\arity(c)} \in \Fml(\sC)$, then $c(\alpha_1, \ldots, \alpha_{\arity(c)}) \in \Fml(\sC)$;
  \item if $\alpha \in \Fml(\sC)$ and $x$ is an individual variable, then $\forall x \alpha \in \Fml(\sC)$ and $\exists x \alpha \in \Fml(\sC)$.
  \end{itemize}
We shall use $\alpha$, $\beta$, $\gamma$, $\varphi$, $\psi$, $\sigma$, $\tau$ and $\chi$ as metavariables for formulas. The set $\FV(\alpha)$ of free variables of $\alpha$ is defined inductively as follows:
  \begin{align*}
  \FV(p(x_1, \ldots, x_n)) & = \{ x_1,\ldots, x_n \}; \\
  \FV(c(\alpha_1, \ldots, \alpha_{\arity(c)})) & = \FV(\alpha_1) \cup \cdots \cup \FV(\alpha_{\arity(c)}); \\
  \FV(\forall x \alpha) = \FV(\exists x \alpha) & = \FV(\alpha) \setminus \{ x \}.
  \end{align*}
A \emph{sequent} is an expression $\Gamma \To \Delta$, where $\Gamma$ and $\Delta$ are sets of formulas. We denote by $\FOSqt(\sC)$ the set $\{ \Gamma \To \Delta \mid \Gamma, \Delta \subseteq \Fml(\sC) \}$. If $\Gamma = \{ \alpha_1, \ldots, \alpha_n \}$ and $\Delta = \{ \beta_1, \ldots, \beta_m \}$, we often omit the braces and simply write $\alpha_1, \ldots, \alpha_n \To \beta_1, \ldots, \beta_m$ for $\{ \alpha_1, \ldots, \alpha_n \} \To \{ \beta_1, \ldots, \beta_m \}$. $\FV(\Gamma \To \Delta)$ denotes the set of free variables of formulas in $\Gamma \cup \Delta$.

Formulas which contain no predicate symbols except propositional symbols are said to be \emph{propositional}. We denote by $\PropFml(\sC)$ the set $\{ \alpha \in \Fml(\sC) \mid \text{$\alpha$ is propositional} \}$ and by $\Sqt(\sC)$ the set 
  \[
  \{ \Gamma \To \Delta \in \FOSqt(\sC) \mid \text{all formulas in $\Gamma \cup \Delta$ are propositional} \}. 
  \]
\subsection{Classical semantics}
\label{Subsection 2.3}
A \emph{(classical) model} $\sM$ is a tuple $\langle D, I \rangle$ in which
  \begin{itemize}
  \item $D$ is a non-empty set, called the \emph{individual domain};
  \item $I$ is a function, called the \emph{interpretation function}, which assigns to each $n$-ary predicate symbol a function from $D^n$ to $\{ 0,1 \}$.
  \end{itemize}
  An \emph{assignment} in $D$ is a function which assigns to each individual variable an element of $D$. For an assignment $\rho$ in $D$, an individual variable $x$ and an element $a \in D$, we write $\rho[x \mapsto a]$ for the assignment in $D$ which maps $x$ to $a$ and is equal to $\rho$ everywhere else.
For a model $\sM = \langle D, I \rangle$, a formula $\alpha \in \Fml(\sC)$ and an assignment in $D$, we define the \emph{interpretation} $\llbracket \alpha \rrbracket_{\sM}^{\rho}$ of $\alpha$ with respect to $\rho$ inductively as follows:
  \begin{itemize}
  \item $\llbracket p(x_1, \ldots, x_n) \rrbracket_{\sM}^{\rho} = I(p)(\rho(x_1), \ldots, \rho(x_n))$;
  \item $\llbracket c(\alpha_1, \ldots, \alpha_{\arity(c)}) \rrbracket_\sM^\rho = \ttfunc{c}(\llbracket \alpha_1 \rrbracket_{\sM}^\rho, \ldots, \llbracket \alpha_{\arity(c)} \rrbracket_\sM^\rho)$;
  \item $\llbracket \forall x \alpha \rrbracket_{\sM}^\rho = 1$ if and only if $\llbracket \alpha \rrbracket_{\sM}^{\rho[x \mapsto a]} = 1$ for all $a \in D$;
  \item $\llbracket \exists x \alpha \rrbracket_{\sM}^\rho = 1$ if and only if $\llbracket \alpha \rrbracket_{\sM}^{\rho[x \mapsto a]} = 1$ for some $a \in D$.
  \end{itemize}
The value of $\llbracket \alpha \rrbracket_{\sM}^\rho$ only depends on the values of $\rho$ on $\FV(\alpha)$. Hence, even for a partial function $\rho$ from the set of individual variables to $D$ whose domain includes $\FV(\alpha)$, $\llbracket \alpha \rrbracket_\sM^\rho$ can be defined to be the value $\llbracket \alpha \rrbracket_\sM^{\rho'}$ for any total function $\rho'$ from the set of individual variables to $D$ which is an extension of $\rho$. We call a partial function from the set of individual variables to an individual domain a \emph{partial assignment} . Even for a partial assignment $\rho$, we define $\rho[x \mapsto a]$ to be the function which maps $x$ to $a$ and is equal to $\rho$ on $\dom(\rho)\setminus \{ x\}$. We use $\emptyfunc$ to denote the empty assignment $\emptyset \to D$. For example, for a model $\langle D, I \rangle$ with $a, b \in D$, we have $\llbracket \bot \rrbracket_{\langle D, I \rangle}^{\emptyfunc} = 0$ and $\llbracket p(x,y) \rrbracket_{\langle D, I \rangle}^{\emptyfunc [x \mapsto a] [y \mapsto b]} = I(p)(a, b)$.

 If $\vec{\alpha}$ denotes a sequence of formulas $\alpha_1, \ldots, \alpha_n$, then we denote by $\llbracket \vec{\alpha} \rrbracket_{\sM}^{\rho}$ the sequence of interpretations of $\alpha_1, \ldots, \alpha_n$, $\langle \llbracket \alpha_1 \rrbracket_\sM^\rho, \ldots, \llbracket \alpha_n \rrbracket_{\sM}^{\rho} \rangle$. For example, if $\alpha \equiv c(\beta_1, \ldots, \beta_{\arity(c)})$ and $\vec{\beta} = \beta_1, \ldots, \beta_{\arity(c)}$, then $\llbracket \alpha \rrbracket_\sM^\rho = 1$ if and only if $\ttfunc{c}(\llbracket \vec{\beta} \rrbracket_\sM^\rho) = 1$.

A formula $\alpha \in \Fml(\sC)$ is \emph{valid} in a classical model $\sM = \langle D, I \rangle$ if $\llbracket \alpha \rrbracket_\sM^\rho = 1$ holds for all assignments $\rho$ in $D$. A formula $\alpha \in \Fml(\sC)$ is \emph{(classically) valid} if it is valid in all classical models. We denote by $\FOCL(\sC)$ the set $\{ \alpha \in \Fml(\sC) \mid \text{$\alpha$ is classically valid} \}$.

For a sequent $\Gamma \To \Delta \in \FOSqt(\sC)$, the \emph{interpretation} $\llbracket \Gamma \To \Delta \rrbracket_{\sM}^{\rho} \in \{0,1\}$ of $\Gamma \To \Delta$ with respect to $\rho$ is defined by
  \[
  \llbracket \Gamma \To \Delta \rrbracket_\sM^\rho =
    \begin{cases}
    0 & \text{if $\llbracket \alpha \rrbracket_\sM^\rho = 1$ for all $\alpha \in \Gamma$ and $\llbracket \beta \rrbracket_\sM^\rho = 0$ for all $\beta \in \Delta$} \\
    1 & \text{otherwise}.
    \end{cases}
  \]
A sequent $\Gamma \To \Delta \in \FOSqt(\sC)$ is \emph{valid} in a classical model $\sM = \langle D, I \rangle$ if $\llbracket \Gamma \To \Delta \rrbracket_\sM^\rho = 1$ holds for all assignments $\rho$ in $D$.
  A sequent $\Gamma \To \Delta \in \FOSqt(\sC)$ is \emph{(classically) valid} if it is valid in all classical models. We denote by $\FOCLS(\sC)$ the set $\{ \Gamma \To \Delta \in \FOSqt(\sC) \mid \text{$\Gamma \To \Delta$ is classically valid} \}$.
\subsection{Kripke semantics}
\label{Subsection 2.4}
A \emph{Kripke model} is a tuple $\langle W, \preceq, D, I \rangle$ in which
  \begin{itemize}
  \item $W$ is a non-empty set, called a set of \emph{possible worlds};
  \item $\preceq$ is a pre-order on $W$;
  \item $D$ is a function that assigns to each $w \in W$ a non-empty set $D(w)$, which is called the \emph{individual domain} at $w$. Furthermore, $D$ satisfies the monotonicity: for all $w, v \in W$, if $w \preceq v$ then $D(w) \subseteq D(v)$.
  \item $I$ is a function, called an \emph{interpretation function}, that assigns to each pair $\langle w, p \rangle$ of a possible world and an $n$-ary predicate symbol a function $I(w,p)$ from $D(w)^n$ to $\{0,1\}$. Furthermore, $I$ satisfies the \emph{hereditary condition}: for all $n$-ary predicate symbols $p$ and all $w, v \in W$, if $w \preceq v$ then $I(w,p) (a_1, \ldots, a_n) \leq I(v, p) (a_1, \ldots, a_n)$ holds for all $a_1, \ldots, a_n \in D(w)$.
  \end{itemize} 
An \emph{assignment} in $D(w)$ is a function which assigns to each individual variable an element of $D(w)$.  For an assignment $\rho$ in $D(w)$, an individual variable $x$ and an element $a \in D$, we write $\rho[x \mapsto a]$ for the assignment in $D(w)$ which maps $x$ to $a$ and is equal to $\rho$ everywhere else.  
For a Kripke model $\sK = \langle W, \preceq, D, I \rangle$, a possible world $w \in W$, an assignment $\rho$ in $D(w)$ and a formula $\alpha \in \Fml(\sC)$, we define the \emph{interpretation} $\| \alpha \|_{\sK, w}^{\rho} \in \{ 0, 1 \}$ of $\alpha$ at $w$ with respect to $\rho$ as follows:
  \begin{itemize}
  \item $\| p (x_1, \ldots, x_n) \|_{\sK, w}^\rho = I(w, p)(\rho(x_1), \ldots, \rho(x_n))$; 
  \item $\| c(\alpha_1, \ldots, \alpha_n) \|_{\sK, w}^\rho = 1$ if and only if $\ttfunc{c}(\| \alpha_1 \|_{\sK, v}^\rho, \ldots, \| \alpha_n \|_{\sK, v}^\rho) = 1$ for all $v \succeq w$;
  \item $\| \forall x \alpha \|_{\sK, w}^\rho = 1$ if and only if $\| \alpha \|_{\sK, v}^{\rho[x \mapsto a]} = 1$ for all $v \succeq w$ and all $a \in D(v)$;
  \item $\| \exists x \alpha \|_{\sK, w}^\rho = 1$ if and only if $\| \alpha \|_{\sK, w}^{\rho [x \mapsto a]} = 1$ for some $a \in D(w)$.
  \end{itemize}
Note that, in case $c = \land$ or $c = \lor$, the statement of the definition of $\| c(\alpha_1, \alpha_2) \|_{\sK, w}^\rho$ differs from the usual one, in which the interpretation is defined by the interpretations of $\alpha_1$ and $\alpha_2$ only at $w$, but we can easily verify that this definition is equivalent to the usual one.   

The value of $\llbracket \alpha \rrbracket_{\sK, w}^\rho$ only depends on the values of $\rho$ on $\FV(\alpha)$. Hence, even for a partial function $\rho$ from the set of individual variables to $D(w)$ whose domain includes $\FV(\alpha)$, $\llbracket \alpha \rrbracket_{\sK, w}^\rho$ can be defined to be  the value $\llbracket \alpha \rrbracket_{\sK, w}^{\rho'}$ for any total function $\rho'$ from the set of individual variables to $D(w)$ which is an extension of $\rho$. We call a partial function from the set of individual variables to an individual domain a \emph{partial assignment} . Even for a partial assignment $\rho$, we define $\rho[x \mapsto a]$ to be the function which maps $x$ to $a$ and is equal to $\rho$ on $\dom(\rho)\setminus \{ x\}$. We use $\emptyfunc$ to denote the empty assignment $\emptyset \to D(w)$. For example, for a Kripke model $\langle W, \preceq, D, I \rangle$, a possible world $w \in W$ and individuals $a, b \in D(w)$, we have $\llbracket \bot \rrbracket_{\sK, w}^{\emptyfunc} = 0$ and $\llbracket p(x,y) \rrbracket_{\sK, w}^{\emptyfunc [x \mapsto a] [y \mapsto b]} = I(w, p)(a, b)$.

 If $\vec{\alpha}$ denotes a sequence of formulas $\alpha_1, \ldots, \alpha_n$, then we denote by $\llbracket \vec{\alpha} \rrbracket_{\sK, w}^{\rho}$ the sequence of interpretations of $\alpha_1, \ldots, \alpha_n$, $\langle \llbracket \alpha_1 \rrbracket_{\sK, w}^\rho, \ldots, \llbracket \alpha_n \rrbracket_{\sK, w}^{\rho} \rangle$. For example, if $\alpha \equiv c(\beta_1, \ldots, \beta_{\arity(c)})$ and $\vec{\beta} = \beta_1, \ldots, \beta_{\arity(c)}$, then $\llbracket \alpha \rrbracket_{\sK, w}^\rho = 1$ if and only if $\ttfunc{c}(\llbracket \vec{\beta} \rrbracket_{\sK, v}^\rho) = 1$ for any $v \succeq w$.

A formula $\alpha \in \Fml(\sC)$ is \emph{valid} in a Kripke model $\sK = \langle W, \preceq, D, I \rangle$ if $\| \alpha \|_{\sK, w}^\rho = 1$ for any $w \in W$ and any assignment $\rho$ in $D(w)$. A formula $\alpha \in \Fml(\sC)$ is \emph{Kripke-valid} if it is valid in all Kripke models. We denote by $\FOIL(\sC)$ the set $\{ \alpha \in \Fml(\sC) \mid \text{$\alpha$ is Kripke-valid} \}$.

As in the case of the usual connectives,  the hereditary condition easily extends to any formula:
  \begin{lemma}\label{Lemma 1}
  For any formula $\alpha \in \Fml(\sC)$, any Kripke model $\sK = \langle W, \preceq, D, I \rangle$, any $w,v \in W$ and any assignment $\rho$ in $D(w)$, if $w \preceq v$ then $\| \alpha \|_{\sK, w}^\rho \leq \| \alpha \|_{\sK, w}^\rho$. 
  \end{lemma}
  We shall use this lemma without references.

For a Kripke model $\sK = \langle W, \preceq, D, I \rangle$, a possible world $w \in W$, an assignment $\rho$ in $D(w)$ and a sequent $\Gamma \To \Delta \in \FOSqt(\sC)$, the \emph{interpretation} $\| \Gamma \To \Delta \|_{\sK, w}^{\rho} \in \{0,1\}$ of $\Gamma \To \Delta$  at $w$ with respect to $\rho$ is defined by
  \[
  \| \Gamma \To \Delta \|_{\sK, w}^\rho =
    \begin{cases}
    0 & \text{if $\| \alpha \|_{\sK, w}^\rho = 1$ for all $\alpha \in \Gamma$ and $\| \beta \|_{\sK, w}^\rho = 0$ for all $\beta \in \Delta$} \\
    1 & \text{otherwise}.
    \end{cases}
  \]
  For a Kripke model $\sK = \langle W, \preceq, D, I \rangle$, a sequent $\Gamma \To \Delta \in \FOSqt(\sC)$ is \emph{valid} in $\sK$ if $\| \Gamma \To \Delta \|_{\sK,w}^\rho = 1$ for all $w \in W$ and all assignment $\rho$ in $D(w)$. A sequent $\Gamma \To \Delta \in \FOSqt(\sC)$ is \emph{Kripke-valid} if it is valid in all Kripke models. We denote by $\FOILS(\sC)$ the set $\{ \Gamma \To \Delta \in \FOSqt(\sC) \mid \text{$\Gamma \To \Delta$ is Kripke-valid} \}$

  A Kripke model $\sK = \langle W, \preceq, D, I \rangle$ is said to be \emph{constant domain} if $D(w) = D(v)$ for all $w, v \in W$. In this case, we simply write $D$ for $D(w)$ for any $w \in W$. Note that, for a constant domain Kripke model $\sK = \langle W, \preceq, D, I \rangle$, the interpretation of an universal formula may be defined only at the present world, that is: $\| \forall x \alpha \|_{\sK, w}^\rho = 1$ if and only if $\| \alpha \|_{\sK, w}^{\rho[x \mapsto a]} = 1$ for all $a \in D$. 
  A formula $\alpha \in \Fml(\sC)$ is \emph{CD-valid} if it is valid in all constant domain Kripke models. We denote by $\FOCD(\sC)$ the set $\{ \alpha \in \Fml(\sC) \mid \text{$\alpha$ is CD-valid} \}$.
  A sequent $\Gamma \To \Delta \in \Sqt(\sC)$ is \emph{CD-valid} if it is valid in all constant domain Kripke models. We denote by $\FOCDS(\sC)$ the set $\{ \Gamma \To \Delta \in \FOSqt \mid \text{$\Gamma \To \Delta$ is CD-valid}\}$.
    
The following lemma follows by the definition of $\FOCDS(\sC)$ and $\FOCLS(\sC)$:
  \begin{lemma}\label{Lemma 2}
  $\FOCDS(\sC) \subseteq \FOCLS(\sC)$ for any set $\sC$ of connectives.
  \end{lemma}
\section{Condition for $\FOCDS(\sC)$ and $\FOCLS(\sC)$ to coincide}
\label{Section 3}
In this section, we show the following theorem:
  \begin{theorem}\label{Theorem 1}
  $\FOCDS(\sC) = \FOCLS(\sC)$ if and only if all connectives in $\sC$ are monotonic.
  \end{theorem}
We show the ``if'' part in \S \ref{Subsection 3.1} and the ``only if'' part in \S \ref{Subsection 3.2}.

\subsection{The ``if'' part}
\label{Subsection 3.1}
Here, we show the ``if'' part of Theorem \ref{Theorem 1}:
  \begin{proposition}\label{Proposition 1}
  If all connectives in $\sC$ are monotonic, then $\FOCDS(\sC) = \FOCLS(\sC)$.
  \end{proposition}
The following lemma is essential for the proof of this proposition.
  \begin{lemma}\label{Lemma 3}
  Suppose all connectives in $\sC$ are monotonic. Let $\sK = \langle W, \preceq, D, I \rangle$ be a constant domain Kripke model and $w \in W$. Let $\sM_{\sK, w} = \langle D, J_{\sK, w} \rangle$ be the classical model defined by $J_{\sK, w}(p) = I(p, w)$. Then, for any formula $\alpha \in \Fml(\sC)$ and any assignment $\rho$ in $D$, $\| \alpha \|_{\sK, w}^\rho = \llbracket \alpha \rrbracket_{\sM_{\sK, w}}^\rho$ holds.
  \end{lemma}
  \begin{proof}
  The proof proceeds by induction on $\alpha$. The base case, in which $\alpha$ is atomic, immediately follows by the definition of $J_{\sK, w}$. Now, we show the inductive step by cases of the form of $\alpha$. 
  
  \verticalspace
  
  \textsc{Case} 1: $\alpha$ is of the form $c(\beta_1, \ldots, \beta_{\arity(c)})$. Put $\vec{\beta} = \beta_1, \ldots, \beta_{\arity(c)}$. By the hereditary, we have $\| \vec{\beta} \|_{\sK, w}^\rho \sqsubseteq \| \vec{\beta} \|_{\sK, v}^\rho$  for all $v \succeq w$. Hence, since $c$ is monotonic, we have $\ttfunc{c}(\| \vec{\beta} \|_{\sK, w}^{\rho}) \leq \ttfunc{c}(\| \vec{\beta} \|_{\sK, v}^{\rho})$ for all $v \succeq w$, so that $\| \alpha \|_{\sK, w}^\rho = \ttfunc{c}(\| \vec{\beta} \|_{\sK, w}^\rho)$ holds. On the other hand, by the induction hypothesis, we have $\ttfunc{c}(\| \vec{\beta} \|_{\sK, w}^\rho) = \ttfunc{c}(\| \vec{\beta} \|_{\sM_{\sK, w}}^\rho) = \llbracket \alpha \rrbracket_{\sM_{\sK, w}}^\rho$.
  
  \verticalspace
  
  \textsc{Case} 2: $\alpha$ is of the form $\forall x \beta$. In this case, we have
    \begin{alignat*}{2}
    \| \alpha \|_{\sK, w}^\rho & = \min_{a \in D} \| \beta \|_{\sK, w}^{\rho [x \mapsto a]} \\
    & = \min_{a \in D} \| \beta \|_{\sM_{\sK, w}}^{\rho [x \mapsto a]} & \quad & \text{(by the induction hypothesis)} \\
    & = \| \alpha \|_{\sM_{\sK, w}}^\rho.
    \end{alignat*}
    
  \verticalspace
    
  \textsc{Case} 3: $\alpha$ is of the form $\exists x \beta$. In this case, we have
    \begin{alignat*}{2}
    \| \alpha \|_{\sK, w}^\rho & = \max_{a \in D} \| \beta \|_{\sK, w}^{\rho [x \mapsto a]} \\
    & = \max_{a \in D} \| \beta \|_{\sM_{\sK, w}}^{\rho [x \mapsto a]} & \quad & \text{(by the induction hypothesis)} \\
    & = \| \alpha \|_{\sM_{\sK, w}}^\rho.
    \end{alignat*}
  \qed  
  \end{proof}
  
Using this lemma, we prove Proposition \ref{Proposition 1}. 
  \begin{proof}[of Proposition \ref{Proposition 1}]
  Suppose all connectives in $\sC$ are monotonic. By Lemma \ref{Lemma 2}, it suffices to show $\FOCLS(\sC) \subseteq \FOCDS(\sC)$. In order to show this inclusion, we suppose $\Gamma \To \Delta \in \FOCLS(\sC)$, and show that $\| \Gamma \To \Delta \|_{\sK, w}^\rho = 1$ holds for any constant domain Kripke model $\sK = \langle W, \preceq, D, I \rangle$, any possible world $w \in W$ and any assignment $\rho$ in $D$. By Lemma \ref{Lemma 3}, it holds that $\| \Gamma \To \Delta \|_{\sK, w}^\rho = \llbracket \Gamma \To \Delta \rrbracket_{\sM_{\sK, w}}^\rho$ for any such $\sK$, $w$ and $\rho$. For any such $\sK$, $w$ and $\rho$, since $\Gamma \To \Delta \in \FOCLS(\sC)$, we have $\llbracket \Gamma \To \Delta \rrbracket_{\sM_{\sK, w}}^\rho = 1$, and hence, we have $\| \Gamma \To \Delta \|_{\sK, w}^\rho = 1$.
  \qed
  \end{proof}
\subsection{The ``only if'' part}
\label{Subsection 3.2}
Here, we show the ``only if'' part of Theorem \ref{Theorem 1} by showing its contraposition:
  \begin{proposition}\label{Proposition 2}
   If $\sC$ has a non-monotonic connective, then $\FOCLS(\sC) \setminus \FOCDS(\sC) \neq \emptyset$.
   \end{proposition}
   
    In \cite{kawano2021effect}, the following corresponding claim was shown in the case of propositional logic:
  \begin{proposition}\label{Proposition 3}
  If $\sC$ has a non-monotonic connective, then $\CLS(\sC) \setminus \ILS(\sC) \neq \emptyset$.
  \end{proposition}
Here, $\ILS(\sC)$ denotes the set of  propositional sequents $\Gamma \To \Delta \in \Sqt(\sC)$ which are valid in all Kripke models for intuitionistic propositional logic and $\CLS(\sC)$ denotes the set of propositional sequents $\Gamma \To \Delta \in \Sqt(\sC)$ which are valid in all models for classical propositional logic. Actually, Proposition \ref{Proposition 2} follows from Proposition \ref{Proposition 3}, because the followings hold:
  \begin{itemize}
  \item For any $\Gamma \To \Delta \in \Sqt(\sC)$, $\Gamma \To \Delta \in \ILS(\sC)$ if and only if $\Gamma \To \Delta \in \FOCDS(\sC)$.
  \item For any $\Gamma \To \Delta \in \Sqt(\sC)$, $\Gamma \To \Delta \in \CLS(\sC)$ if and only if $\Gamma \To \Delta \in \FOCLS(\sC)$.
  \end{itemize}
However, for the purpose of self-containedness, here we describe the direct proof.
  \begin{proof}[of Proposition \ref{Proposition 2}]
  We show that if $\sC$ includes a non-monotonic connective, then $\FOCLS(\sC) \setminus \FOCDS(\sC) \neq \emptyset$. We fix distinct propositional symbols $p$, $q$, $r$ and $s$.
  
  Let $c$ be a non-monotonic connective in $\sC$. We divide into four cases: $(\mathrm{a})$ $\ttfunc{c}(\bzero) = \ttfunc{c}(\bone) = 0$; $(\mathrm{b})$ $\ttfunc{c}(\bzero) = 0$ and $\ttfunc{c}(\bone) = 1$; $(\mathrm{c})$ $\ttfunc{c}(\bzero) = 1$ and $\ttfunc{c}(\bone) = 0$; and $(\mathrm{d})$: $\ttfunc{c}(\bzero) = \ttfunc{c}(\bone) = 1$. We show in the order of $(\mathrm{d})$, $(\mathrm{c})$, $(\mathrm{b})$, $(\mathrm{a})$.
  
  \verticalspace
  
  \textsc{Case} $(\mathrm{d})$: $\ttfunc{c}(\bzero) = \ttfunc{c}(\bone) = 1$. First, we construct a formula $\tau$ in $\FOCD(\sC)$. We define $\tau \in \PropFml(\sC)$ by $\tau \equiv c(s, \ldots, s)$. Then, $\tau \in \FOIL(\sC) \subseteq \FOCD(\sC)$ can be easily verified.
  
  Now, we construct a formula $\varphi \in \FOCL(\sC) \setminus \FOCD(\sC)$. We can see, if such $\varphi$ exists, then ${} \To \varphi \in \FOCLS(\sC) \setminus \FOCDS(\sC)$ holds. Since $c$ is non-monotonic, there exist $\ba, \bb \in \{ 0, 1 \}^{\arity(c)}$ such that $\ba \sqsubseteq \bb$, $\ttfunc{c}(\ba) = 1$ and $\ttfunc{c}(\bb) = 0$. Let $\overline{\bb}^{\ba}$ be the sequence in $\{ 0, 1 \}^{\arity(c)}$ defined by
    \[
    \overline{\bb}^{\ba} =
      \begin{cases}
      0 & \text{if $\ba[i] = 0$ and $\bb[i] = 1$} \\
      1 & \text{if $\ba[i] = 1$ or $\bb[i] = 0$}.
      \end{cases}
    \]
  We divide into two subcases: (\textsc{Subcase} 1) $\ttfunc{c}(\overline{\bb}^\ba) = 1$; and (\textsc{Subcase} 2) $\ttfunc{c}(\overline{\bb}^\ba) = 0$.
  
  \textsc{Subcase} 1: $\ttfunc{c}(\overline{\bb}^\ba) = 1$. We define formulas $\sigma^{\mathrm{P}}_1, \ldots, \sigma^{\mathrm{P}}_{\arity(c)}, \sigma^{\mathrm{P}}\in \PropFml(\sC)$ as follows:
    \begin{align*}
    \sigma^{\mathrm{P}}_i & \equiv
      \begin{cases}
      q & \text{if $\ba[i] = 0$ and $\bb[i] = 0$} \\
      p & \text{if $\ba[i] = 0$ and $\bb[i] = 1$} \\
      \tau & \text{if $\ba[i] = 1$} 
      \end{cases} \\
    \sigma^{\mathrm{P}} & \equiv c(\sigma^{\mathrm{P}}_1, \ldots, \sigma^{\mathrm{P}}_{\arity(c)})  
    \end{align*}
  Then, we define formulas $\psi^{\mathrm{P}}_1, \ldots, \psi^{\mathrm{P}}_{\arity(c)}, \psi^{\mathrm{P}} \in \PropFml(\sC)$ as follows:    
    \begin{align*}
    \psi^{\mathrm{P}}_i & \equiv
      \begin{cases}
      p & \text{if $\ba[i] = 0$ and $\bb[i] = 0$} \\
      \sigma^{\mathrm{P}} & \text{if $\ba[i] = 0$ and $\bb[i] = 1$} \\
      \tau & \text{if $\ba[i] = 1$}
      \end{cases} \\
    \psi^{\mathrm{P}} & \equiv c(\psi^{\mathrm{P}}_1, \ldots, \psi^{\mathrm{P}}_{\arity(c)})  
    \end{align*}
  Furthermore, we define formulas $\varphi^{\mathrm{P}}_1, \ldots, \varphi^{\mathrm{P}}_{\arity(c)}, \varphi^{\mathrm{P}} \in \PropFml(\sC)$ as follows:
    \begin{align*}
    \varphi^{\mathrm{P}}_i & \equiv
      \begin{cases}
      p & \text{if $\ba[i] = 0$ and $\bb[i]  = 0$} \\
      \psi^{\mathrm{P}} & \text{if $\ba[i] = 0$ and $\ba[i] = 1$} \\
      \tau & \text{if $\ba[i] = 1$}
      \end{cases} \\
    \varphi^{\mathrm{P}} & \equiv c(\varphi^{\mathrm{P}}_1, \ldots, \varphi^{\mathrm{P}}_{\arity(c)})  
   \end{align*}
    
  Then, we obtain $\varphi^{\mathrm{P}} \in \FOCL(\sC)$ from the following table.
    \begin{center}
      \begin{tabular}{|C{5mm}|C{5mm}||C{26mm}|C{6mm}|C{26mm}|C{6mm}|C{26mm}|C{6mm}|}
       \hline 
      $p$ & $q$ & $\langle \sigma^{\mathrm{P}}_1, \ldots, \sigma^{\mathrm{P}}_{\arity(c)} \rangle$ & $\sigma^{\mathrm{P}}$ & $\langle \psi^{\mathrm{P}}_1, \ldots, \psi^{\mathrm{P}}_{\arity(c)} \rangle$ & $\psi^{\mathrm{P}}$ & $\langle \varphi^{\mathrm{P}}_1, \ldots, \varphi^{\mathrm{P}}_{\arity(c)} \rangle$ & $\varphi^{\mathrm{P}}$ \\ \hline
      $0$ & $0$ & $\ba$ & $1$ & $\bb$ & $0$ & $\ba$ & $1$ \\ \hline
      $0$ & $1$ & $\overline{\bb}^\ba$ & $1$ & $\bb$ & $0$ & $\ba$ & $1$ \\ \hline
      $1$ & $0$ & $\bb$ & $0$ & $\overline{\bb}^\ba$ & $1$ & $\bone$ & $1$ \\ \hline
      $1$ & $1$ & $\bone$ & $1$ & $\bone$ & $1$ & $\bone$ & $1$ \\ \hline
      \end{tabular}
    \end{center}
  Now, consider the constant domain Kripke model $\sK^* = \langle \{w_0, w_1 \}, \preceq, \{ a_1 \}, I \rangle$ in which
    \begin{itemize}
    \item $w_i \preceq w_j$ if and only if $i \leq j$;
    \item $I(w_0, p) = 0$, $I(w_0, q) = 0$, $I(w_1, p) = 1$, and $I(w_1, q) = 0$. (The interpretations for the other pairs of possible worlds and  predicate symbols may be arbitrary.)
    \end{itemize}  
  Then, we obtain $\| \varphi^{\mathrm{P}} \|_{\sK^*, w_0}^\varnothing = 0$ from the following table. For example, that the element in the second row and fourth column is $\overline{\bb}^\ba$ means that
      \[
      \langle \| \psi^{\mathrm{P}}_1 \|_{\sK^*, w_1}^\varnothing, \ldots, \| \psi^{\mathrm{P}}_{\arity(c)} \|_{\sK^*, w_1}^\varnothing \rangle = \overline{\bb}^\ba.
      \]
    \begin{center}
      \begin{tabular}{|c||C{26mm}|C{6mm}|C{26mm}|C{6mm}|C{26mm}|C{6mm}|}
       \hline 
       & $\langle \sigma^{\mathrm{P}}_1, \ldots, \sigma^{\mathrm{P}}_{\arity(c)} \rangle$ & $\sigma^{\mathrm{P}}$ & $\langle \psi^{\mathrm{P}}_1, \ldots, \psi^{\mathrm{P}}_{\arity(c)} \rangle$ & $\psi^{\mathrm{P}}$ & $\langle \varphi^{\mathrm{P}}_1, \ldots, \varphi^{\mathrm{P}}_{\arity(c)} \rangle$ & $\varphi^{\mathrm{P}}$ \\ \hline
      $\| \cdot \|_{\sK^*, w_1}^\varnothing$ & $\bb$ & $0$ & $\overline{\bb}^\ba$ & $1$ & $\bone$ & $1$ \\ \hline
      $\| \cdot \|_{\sK^*, w_0}^\varnothing$ & $\ba$ & $0$ & $\ba$ & $1$ & $\bb$ & $0$ \\ \hline
      \end{tabular}
    \end{center}  
Hence, $\varphi^{\mathrm{P}} \in \FOCL(\sC) \setminus \FOCD(\sC)$.
   
   \textsc{Subcase} 2: $\ttfunc{c}(\overline{\bb}^\ba) = 0$. We define formulas $\sigma^{\mathrm{Q}}_1, \ldots, \sigma^{\mathrm{Q}}_{\arity(c)}, \sigma^{\mathrm{Q}}\in \PropFml(\sC)$ as follows:
    \begin{align*}
    \sigma^{\mathrm{Q}}_i & \equiv
      \begin{cases}
      q & \text{if $\ba[i] = 0$ and $\bb[i] = 0$} \\
      p & \text{if $\ba[i] = 0$ and $\bb[i] = 1$} \\
      \tau & \text{if $\ba[i] = 1$} 
      \end{cases} \\
    \sigma^{\mathrm{Q}} & \equiv c(\sigma^{\mathrm{Q}}_1, \ldots, \sigma^{\mathrm{Q}}_{\arity(c)})  
    \end{align*}
  Then, we define formulas $\psi^{\mathrm{Q}}_1, \ldots, \psi^{\mathrm{Q}}_{\arity(c)}, \psi^{\mathrm{Q}} \in \PropFml(\sC)$ as follows:    
    \begin{align*}
    \psi^{\mathrm{Q}}_i & \equiv
      \begin{cases}
      \sigma^{\mathrm{Q}} & \text{if $\ba[i] = 0$ and $\bb[i] = 0$} \\
      q & \text{if $\ba[i] = 0$ and $\bb[i] = 1$} \\
      \tau & \text{if $\ba[i] = 1$}
      \end{cases} \\
    \psi^{\mathrm{Q}} & \equiv c(\psi^{\mathrm{Q}}_1, \ldots, \psi^{\mathrm{Q}}_{\arity(c)})  
    \end{align*}
  Furthermore, we define formulas $\varphi^{\mathrm{Q}}_1, \ldots, \varphi^{\mathrm{Q}}_{\arity(c)}, \varphi^{\mathrm{Q}} \in \PropFml(\sC)$ as follows:
    \begin{align*}
    \varphi^{\mathrm{Q}}_i & \equiv
      \begin{cases}
      \psi^{\mathrm{Q}} & \text{if $\ba[i] = 0$ and $\bb[i]  = 0$} \\
      p & \text{if $\ba[i] = 0$ and $\ba[i] = 1$} \\
      \tau & \text{if $\ba[i] = 1$}
      \end{cases} \\
    \varphi^{\mathrm{Q}} & \equiv c(\varphi^{\mathrm{Q}}_1, \ldots, \varphi^{\mathrm{Q}}_{\arity(c)})  
   \end{align*}
    
  Then, we obtain $\varphi^{\mathrm{Q}} \in \FOCL(\sC)$ from the following table.
    \begin{center}
      \begin{tabular}{|C{5mm}|C{5mm}||C{26mm}|C{6mm}|C{26mm}|C{6mm}|C{26mm}|C{6mm}|}
       \hline 
      $p$ & $q$ & $\langle \sigma^{\mathrm{Q}}_1, \ldots, \sigma^{\mathrm{Q}}_{\arity(c)} \rangle$ & $\sigma^{\mathrm{Q}}$ & $\langle \psi^{\mathrm{Q}}_1, \ldots, \psi^{\mathrm{Q}}_{\arity(c)} \rangle$ & $\psi^{\mathrm{Q}}$ & $\langle \varphi^{\mathrm{Q}}_1, \ldots, \varphi^{\mathrm{Q}}_{\arity(c)} \rangle$ & $\varphi^{\mathrm{Q}}$ \\ \hline
      $0$ & $0$ & $\ba$ & $1$ & $\overline{\bb}^\ba$ & $0$ & $\ba$ & $1$ \\ \hline
      $0$ & $1$ & $\overline{\bb}^\ba$ & $0$ & $\bb$ & $0$ & $\ba$ & $1$ \\ \hline
      $1$ & $0$ & $\bb$ & $0$ & $\ba$ & $1$ & $\bone$ & $1$ \\ \hline
      $1$ & $1$ & $\bone$ & $1$ & $\bone$ & $1$ & $\bone$ & $1$ \\ \hline
      \end{tabular}
    \end{center}
    
  On the other hand, we obtain $\| \varphi^{\mathrm{Q}} \|_{\sK^*, w_0}^\varnothing = 0$ from the following table.
  \begin{center}
      \begin{tabular}{|c||C{26mm}|C{6mm}|C{26mm}|C{6mm}|C{26mm}|C{6mm}|}
       \hline 
       & $\langle \sigma^{\mathrm{Q}}_1, \ldots, \sigma^{\mathrm{Q}}_{\arity(c)} \rangle$ & $\sigma^{\mathrm{Q}}$ & $\langle \psi^{\mathrm{Q}}_1, \ldots, \psi^{\mathrm{Q}}_{\arity(c)} \rangle$ & $\psi^{\mathrm{Q}}$ & $\langle \varphi^{\mathrm{Q}}_1, \ldots, \varphi^{\mathrm{Q}}_{\arity(c)} \rangle$ & $\varphi^{\mathrm{Q}}$ \\ \hline
      $\| \cdot \|_{\sK^*, w_1}^\varnothing$ & $\bb$ & $0$ & $\ba$ & $1$ & $\bone$ & $1$ \\ \hline
      $\| \cdot \|_{\sK^*, w_0}^\varnothing$ & $\ba$ & $0$ & $\ba$ & $1$ & $\overline{\bb}^\ba$ & $0$ \\ \hline
      \end{tabular}
    \end{center}  

Hence, $\varphi^{\mathrm{Q}} \in \FOCL(\sC) \setminus \FOCD(\sC)$.  
  
  \verticalspace
  
  \textsc{Case} $(\mathrm{c})$: $\ttfunc{c}(\bzero) = 1$ and $\ttfunc{c}(\bone) = 0$. First, we define formula $\lnot_c \alpha$ for each formula $\alpha \in \Fml(\sC)$ by $\lnot_c \alpha \equiv c(\alpha, \ldots, \alpha)$. Then, $\lnot_c \alpha$ plays the same role as $\lnot \alpha$, that is, for any Kripke model $\sK = \langle W, \preceq, D, I \rangle$, any $w \in W$ and any assignment $\rho$ in $D(w)$, $\| \lnot_c \alpha \|_{\sK, w}^\rho = 1$ if and only if $\| \alpha \|_{\sK, v}^\rho = 0$ for all $v \succeq w$. Fix a predicate symbol $p$. Then, it is easy to verify that $\lnot_c \lnot_c p \To p \in \FOCLS(\sC) \setminus \FOCDS(\sC)$.
  
  \verticalspace
 
  \textsc{Case} $(\mathrm{b})$: $\ttfunc{c}(\bzero) = 0$ and $\ttfunc{c}(\bone) = 1$. Since $c$ is non-monotonic, there exist $\ba, \bb \in \{ 0, 1 \}^{\arity(c)}$ such that $\ba \sqsubseteq \bb$, $\ttfunc{c}(\ba) = 1$ and $\ttfunc{c}(\bb) = 0$. We divide into two subcases: (\textsc{Subcase} 1) $\ttfunc{c}(\overline{\ba}) = 1$; and (\textsc{Subcase} 2) $\ttfunc{c}(\overline{\ba}) = 0$.
  
  \textsc{Subcase} 1: $\ttfunc{c}(\overline{\ba}) = 1$. We define formulas $\chi_1, \ldots, \chi_{\arity(c)}, \chi \in \PropFml(\sC)$ as follows:
    \begin{align*}
    \chi_i & \equiv
      \begin{cases}
      q & \text{if $\ba[i] = 0$} \\
      p & \text{if $\ba[i] = 1$}
      \end{cases} \\
    \chi & \equiv c(\chi_1, \ldots, \chi_{\arity(c)})
    \end{align*}
  Then, we can easily verify that, for any model $\sM = \langle D, I \rangle$, if $I(p) = 1$ or $I(q) = 1$, then $\llbracket \chi \rrbracket_\sM^\varnothing = 1$.
  
  Now, we define formulas $\psi_1, \ldots, \psi_{\arity(c)}, \psi \in \PropFml(\sC)$ as follows:
    \begin{align*}
    \psi_i & \equiv
      \begin{cases}
      q & \text{if $\ba[i] = 0$ and $\bb[i] = 0$} \\
      p & \text{if $\ba[i] = 0$ and $\bb[i] = 1$} \\
      r & \text{if $\ba[i] = 1$ and $\bb[i] = 1$}
      \end{cases} \\
    \psi & \equiv c(\psi_1, \ldots, \psi_{\arity(c)})
    \end{align*}
    Then, we can easily verify that, for any model $\sM = \langle D, I \rangle$, $I(p) = I(q) = 0$ implies $\llbracket \psi \rrbracket_\sM^\varnothing = I(r)$.
    
  Next, we define formulas $\varphi_1, \ldots, \varphi_{\arity(c)}, \varphi \in \PropFml(\sC)$ as follows:
    \begin{align*}
    \varphi_i & \equiv
      \begin{cases}
      q & \text{if $\ba[i] = 0$ and $\bb[i] = 0$} \\
      \psi & \text{if $\ba[i] = 0$ and $\bb[i] = 1$} \\
      r & \text{if $\ba[i] = 1$ and $\bb[i] = 1$}
      \end{cases} \\
    \varphi & \equiv c(\varphi_1, \ldots, \varphi_{\arity(c)})  
    \end{align*}
  Then, we can see that, for any model $\sM = \langle D, I \rangle$, if $I(p) = I(q) = 0$ then $\llbracket \varphi \rrbracket_\sM^\varnothing = 0$.
    
  From the above observation, we obtain $\varphi \To \chi \in \FOCLS(\sC)$. Now, let $\sK^+ = \langle \{ w_0, w_1 \}, \preceq, \{ a_1 \}, I \rangle$ be the constant domain Kripke model defined as follows:
    \begin{itemize}
    \item $w_i \preceq w_j$ if and only if $i \leq j$;
    \item $I(w_0, p) = 0$, $I(w_0, q) = 0$, $I(w_0, r) = 1$, $I(w_1, p) = 1$, $I(w_1, q) = 0$, $I(w_1, r) = 1$.
    \end{itemize} 
  Then, from the following table, we obtain $\| \varphi \|_{\sK^+, w_0}^\varnothing = 1$ and $\| \chi \|_{\sK^+, w_0}^\varnothing = 0$. Hence, $\varphi \To \chi \notin \FOCDS(\sC)$.
    \begin{center}
      \begin{tabular}{|c||C{26mm}|C{6mm}|C{26mm}|C{6mm}|C{26mm}|C{6mm}|}
       \hline 
       & $\langle \chi_1, \ldots, \chi_{\arity(c)} \rangle$ & $\chi$ & $\langle \psi_1, \ldots, \psi_{\arity(c)} \rangle$ & $\psi$ & $\langle \varphi_1, \ldots, \varphi_{\arity(c)} \rangle$ & $\varphi$ \\ \hline
      $\| \cdot \|_{\sK^+, w_1}^\varnothing$ & $\ba$ & $1$ & $\bb$ & $0$ & $\ba$ & $1$ \\ \hline
      $\| \cdot \|_{\sK^+, w_0}^\varnothing$ & $\bzero$ & $0$ & $\ba$ & $0$ & $\ba$ & $1$ \\ \hline
      \end{tabular}
    \end{center}  
    
  \textsc{Subcase} 2: $\ttfunc{c}(\overline{\ba}) = 0$. We define $\psi_1, \ldots, \psi_{\arity(c)}, \psi \in \PropFml(\sC)$ as follows:
    \begin{align*}
    \psi_i & \equiv
      \begin{cases}
      q & \text{if $\ba[i] = 0$} \\
      r & \text{if $\ba[i] = 1$}
      \end{cases} \\
    \psi & \equiv c(\psi_1, \ldots, \psi_{\arity(c)})  
    \end{align*}
  Then, we can easily verify that, for any model $\sM = \langle D, I \rangle$,  $I(r) = 0$ implies $\llbracket \psi \rrbracket_\sM^\varnothing = 0$.
  
  Now, let $\varphi^{\mathrm{PP}}$ be the formula obtained from $\varphi^{\mathrm{P}}$ in subcase 1 of case $(\mathrm{d})$ by replacing every occurrence of $\tau$ with $r$. Let $\varphi^{\mathrm{QQ}}$ be the formula obtained from $\varphi^{\mathrm{Q}}$ in subcase 2 of case $(\mathrm{d})$ by replacing every occurrence of $\tau$ with $r$. Then, similarly to case $(\mathrm{d})$, we obtain either $\psi \To \varphi^{\mathrm{PP}} \in \FOCLS(\sC) \setminus \FOCDS(\sC)$ or $\psi \To \varphi^{\mathrm{QQ}} \in \FOCLS(\sC) \setminus \FOCDS(\sC)$. Hence, $\FOCLS(\sC) \setminus \FOCDS(\sC) \neq \emptyset$.
  
  \verticalspace
  
  \textsc{Case} $(\mathrm{a})$: $\ttfunc{c}(\bzero) = \ttfunc{c}(\bone) = 0$. Since $c$ is non-monotonic, there exists some $\ba \in \{ 0, 1 \}^{\arity(c)}$ such that $\ttfunc{c}(\ba) = 1$.
  
  We define formulas $\psi_1, \ldots, \psi_{\arity(c)}, \psi \in \PropFml(\sC)$ as follows:
    \begin{align*}
    \psi_i & \equiv
      \begin{cases}
      p & \text{if $\ba[i] = 0$} \\
      r & \text{if $\bb[i] = 1$}
      \end{cases} \\
    \psi & \equiv c(\psi_1, \ldots, \psi_{\arity(c)})  
    \end{align*}
  Then, we can easily verify that, for any model $\sM = \langle D, I \rangle$, if $I(p) = 0$ then $\llbracket \psi \rrbracket_\sM^\varnothing = I(r)$.  
  
  Now, we define formulas $\varphi_1, \ldots, \varphi_{\arity(c)}, \varphi \in \PropFml(\sC)$ as follows:
    \begin{align*}
    \varphi_i & \equiv
      \begin{cases}
      \psi & \text{if $\ba[i] = 0$} \\
      r & \text{if $\ba[i] = 1	$}
      \end{cases} \\
    \varphi & \equiv c(\varphi_1, \ldots, \varphi_{\arity(c)})
    \end{align*}
  Then, we can easily verify that, for any model $\sM = \langle D, I \rangle$, if $I(p) = 0$ then $\llbracket \varphi \rrbracket_\sM^\varnothing = 0$. Hence, we obtain $\varphi \To p \in \FOCLS(\sC)$.
  
  On the other hand, for the constant domain Kripke model $\sK^+$ given in case $(\mathrm{b})$, we have $\| p \|_{\sK^+, w_0}^\varnothing = 1$, and we obtain $\| \varphi \|_{\sK^+, w_0}^\varnothing = 1$ from the following table. Hence, we have $\varphi \To p \notin \FOCDS(\sC)$.
    \begin{center}
      \begin{tabular}{|c||C{26mm}|C{6mm}|C{26mm}|C{6mm}|}
       \hline 
        & $\langle \psi_1, \ldots, \psi_{\arity(c)} \rangle$ & $\psi$ & $\langle \varphi_1, \ldots, \varphi_{\arity(c)} \rangle$ & $\varphi$ \\ \hline
      $\| \cdot \|_{\sK^+, w_1}^\varnothing$ & $\bone$ & $0$ & $\ba$ & $1$ \\ \hline
      $\| \cdot \|_{\sK^+, w_0}^\varnothing$ & $\ba$ & $0$ & $\ba$ & $1$ \\ \hline
      \end{tabular}
    \end{center}  
  \qed  
  \end{proof}
\section{Conclusion}
\label{Section 4}
We have seen that generalized Kripke semantics can be extended to first-order logic. Furthermore, if we only admit as models Kripke models with constant domains, then we obtain constant domain Kripke semantics that admits general propositional connectives. Then, extending the the theorem that gives the necessary and sufficient condition for $\ILS(\sC)$ and $\CLS(\sC)$ to coincide, we have obtained the following theorem:
  \begin{theorem*}
  $\FOCDS(\sC) = \FOCLS(\sC)$ if and only if all connectives in $\sC$ are monotonic.
  \end{theorem*} 
 
\bibliographystyle{splncs04}
\bibliography{main}
\end{document}